\documentclass[12pt]{amsart}
\usepackage{amsmath,amssymb,amsthm}
\usepackage{amsfonts}
\usepackage[mathscr]{eucal}
\usepackage{color}
\newtheorem{theorem}{Theorem}[section]
\newtheorem{lemma}[theorem]{Lemma}

\newtheorem{corollary}[theorem]{Corollary}

\theoremstyle{definition}

\theoremstyle{remark}
\newtheorem{remark}[theorem]{Remark}

\numberwithin{equation}{section}

\newcommand{\RR}[1]{\mathbb{#1}}

\newcommand{\rd}{{\mathbb R^d}}

\newcommand{\rr}{{\mathbb R}}

\def\N{{\mathbb N}}

\def\E{{\mathbb E}}

\def\D{{\mathbb D}}

\begin{document}

\title{\bf Distributed-order fractional Cauchy problems on bounded domains}

\author{Mark M. Meerschaert}
\address{Mark M. Meerschaert, Department of Statistics and Probability,
Michigan State University, East Lansing, Michigan 48823, USA.}
\email{mcubed@stt.msu.edu}
\urladdr{http://www.stt.msu.edu/$\sim$mcubed/}
\thanks{The research of MMM was partially supported by NSF grants DMS-0803360 and EAR-0823965.}

\author{Erkan Nane}
\address{Erkan Nane, Department of Mathematics and Statistics, 221 Parker Hall, Auburn University, Auburn, Alabama 36849, USA}
\email{nane@auburn.edu}
\thanks{EN is the corresponding author.}

\author{P. Vellaisamy}
\address{P. Vellaisamy, Department of Mathematics,
 Indian Institute of Technology Bombay, Powai, Mumbai 400076, INDIA.}
\email{pv@math.iitb.ac.in}


\begin{abstract}
In a fractional Cauchy problem, the usual first order time
derivative is replaced by a fractional derivative.  The fractional
derivative models time delays in a diffusion process.  The order
of the fractional derivative can be distributed over the unit
interval, to model a mixture of delay sources.   In this paper, we
provide explicit strong solutions and stochastic analogues for
distributed-order fractional Cauchy problems on bounded domains
with Dirichlet boundary conditions. Stochastic solutions are
constructed using a non-Markovian time change of a killed Markov
process generated by a uniformly elliptic second order space
derivative operator.
\end{abstract}


\maketitle

\newpage

\section{Introduction}
This paper develops explicit strong solutions for
distributed-order fractional Cauchy problems on bounded domains
$D\subset \rd$ with Dirichlet boundary conditions.  Cauchy
problems $\partial u/\partial t=Lu$ model diffusion processes.
The simplest case $L=\Delta=\sum_j \partial^2 u/\partial x_j^2$
governs a Brownian motion $B(t)$ with
 density $u(t,x)$, for which the square root scaling $u(t,x)=t^{-1/2}u(1,t^{-1/2}x)$ pertains \cite{Einstein1906}.
   The fractional Cauchy problem $\partial^\beta u/\partial t^\beta=Lu$ with $0<\beta<1$ models anomalous
   sub-diffusion, in which a cloud of particles spreads slower than the square root of time
   \cite{koch1,koch2,nigmatullin,sch-wyss,zaslavsky}.  When $L=\Delta$, the solution $u(t,x)$ is the density of
    a time-changed Brownian motion $B(E_t)$, where the non-Markovian time change $E_t=\inf\{\tau>0:W_\tau>t\}$ is the
    inverse, or first passage time, of a stable subordinator $W_t$ with index $\beta$.  The scaling
    $W_{ct}=c^{1/\beta}W_t$ in law implies $E_{ct}=c^\beta E_t$ in law for the inverse process,
    so that $u(t,x)=t^{-\beta/2}u(1,t^{-\beta/2}x)$.  The process $B(E_t)$ is the long-time scaling
    limit of a random walk  \cite{Zsolution,limitCTRW} when the random waiting times between
    jumps belong to the $\beta$-stable domain of attraction.  Roughly speaking, a power-law distribution
    of waiting times leads to a fractional time derivative in the governing equation, and its  power law
    index
    equals the order of the fractional time derivative.
    For a uniformly elliptic operator $L$ on a bounded domain $D\subset \rd$, under suitable technical
    conditions and assuming Dirichlet boundary conditions, the Cauchy problem governs a Markov process
     $X(t)$ killed at the boundary.  Then the fractional Cauchy problem governs the time-changed process
     $X(E_t)$ \cite{m-n-v-aop}. Recently, Barlow and \u Cern\'y \cite{barlow-cerny} obtained $B(E_t)$ as
     the scaling limit of a random walk in a random environment, when the transition rates are drawn from a
     power-law distribution with index $\beta$.

Fractional derivatives are almost as old as their
integer-order cousins \cite{MillerRoss, Samko}.  Fractional
diffusion equations are important in
physics, finance, hydrology, and many other areas
\cite{GorenfloSurvey,koch2,MetzlerKlafter,scalas1}.  Nigmatullin \cite{nigmatullin} gave a
physical derivation of the fractional Cauchy problem.  The mathematical study of fractional Cauchy problems was
initiated by Schneider and Wyss \cite{sch-wyss} and Kochubei \cite{koch1, koch2}.
Fractional Cauchy problems were used by
Zaslavsky \cite{zaslavsky} as a model for Hamiltonian chaos.
Stochastic solutions of fractional Cauchy problems are the basis for particle tracking schemes
\cite{Chakraborty,Chechkin,time-Langevin}.

In some applications, the waiting times between particle jumps
evolve according to a more complicated process that cannot be
adequately described by a single power law.  Then, a waiting time
model that is conditionally power law leads to a distributed-order
fractional derivative in time, defined by integrating the
fractional derivative of order $\beta$ against the probability
distribution of the power-law index \cite{M-S-ultra}. The
resulting distributed-order fractional Cauchy problem provides a
more flexible model for anomalous sub-diffusion.  The L\'evy
measure of a stable subordinator with index $\beta$ is integrated
against the power law index distribution to define a subordinator
$W_t$.  Its inverse $E_t$ produces a stochastic solution $X(E_t)$
of the distributed-order fractional Cauchy problem on $\rd$, when
$X(t)$ solves the original
 Cauchy problem $\partial u/\partial t=Lu$.  Mild solutions and stochastic solutions for distributed-order
  fractional Cauchy problems on $\rd$ were developed by Meerschaert and Scheffler \cite{M-S-ultra},
  Kov\'acs and Meerschaert \cite{ultrafast}, and Hahn, Kobayashi and Umarov \cite{h-k-umarov}.
   Kochubei \cite{koch3} obtained strong solutions for the case $L=\Delta$, and proved the uniqueness for more general
    uniformly elliptic operators $L$.

In this paper, we extend the basic approach of \cite{Zsolution,limitCTRW} to solve distributed-order fractional
 Cauchy problems with Dirichlet boundary conditions on bounded domains $D\subset \rd$.  Our recent paper
  \cite{m-n-v-aop} treated fractional Cauchy problems on bounded domains using eigenfunction expansions
  and killed Markov processes.  Here, we extend that theory to the distributed-order fractional Cauchy problems,
  using the inverse subordinators from \cite{ultrafast,M-S-ultra}, along with some deep technical results
  from Kochubei \cite{koch3}.  We construct unique classical solutions, identify the stochastic process
  governed by the distributed order fractional Cauchy problem, and prove existence.  At the end of this paper,
   we  discuss also some open problems in the literature, including extensions to jump processes.

\section{Heat kernels on bounded domains}\label{sec2}

Let $D$ be a bounded domain in $\rd$. We denote by $C^k(D),
C^{k,\alpha}(D), C^k(\bar D)$  the space of k-times differentiable
functions in $D$, the space of $k$-times differential functions
with $k$-th derivative is H\"older continuous of index $\alpha$,
 and the space of functions that have all the derivatives up to order  $k$ extendable continuously up to the
 boundary $\partial D$ of $D$, respectively. We refer to  \cite{m-n-v-aop} for a detailed discussion of
 these spaces and concepts in this section.

A uniformly elliptic operator in divergence form is defined on
$C^2$ functions by
\begin{equation}\label{unif-elliptic-op}
Lu=\sum_{i,j=1}^{d}\frac{\partial \left(a_{ij}(x)(\partial
u/\partial x_i)\right)}{\partial x_j}
\end{equation}
with $a_{ij}(x)=a_{ji}(x)$ and, for some $\lambda>0,$
\begin{equation}\label{elliptic-bounds}
\lambda \sum_{i=1}^ny_i^2\leq \sum_{i,j=1}^na_{ij}(x)y_iy_j\leq
\lambda^{-1} \sum_{i=1}^ny_i^2,\ \ \forall y \in \rd.
\end{equation}

The operator $L$ acts on the Hilbert space $L^2(D)$. We define the
initial domain $C_0^\infty(\bar D)$ of the operator as follows. We
say that $f$ is in $C_0^\infty(\bar D)$ if  $f\in C^\infty(\bar
D)$ and $f(x)=0$ for all $x\in \partial D$. This condition
incorporates the notion of  Dirichlet boundary conditions.

If $X_t$ is a solution to
$$
dX_t=\sigma (X_t)dB_t+b(X_t)dt, \ \ X_0=x_0,
$$
where $\sigma$ is a $d\times d$ matrix and $B_t$ is a Brownian
motion, then $X_t$ is associated with the operator $L$ with
$a=\sigma \sigma ^T$ (see Chapters 1 and 5 of Bass \cite{bass}).
Define the first exit time as
 $\tau_D(X)=\inf \{ t\geq 0:\ X_t\notin D\}$. Then the semigroup defined
 by
 \begin{equation}\label{eqn2.3n}
T_{D}(t)f(x)=E_x[f(X_t)I(\tau_D(X))>t)]
\end{equation}
 has generator $L$, which
follows by an application of the It$\mathrm{\hat{ o}}$ formula.









 Let $D $ be a  bounded domain in $\RR{R}^d$ and  $L$ be a uniformly elliptic operator of divergence form
with Dirichlet
  boundary conditions on $D$. Then  there exists a constant $\Lambda$ such that for all $x\in D$,
\begin{equation}\label{L-uniform-bound}
\sum_{i,j=1}^d |a_{ij}(x)|\leq \Lambda,
\end{equation}
 see, for example, Davies \cite[Chapter 6]{davies-d}

Let
  $T_D(t)$ be the corresponding semigroup. Then
 $T_D(t)$ is an ultracontractive semigroup (even  an intrinsically
 ultracontractive semigroup), see Corollary 3.2.8, Theorem 2.1.4, Theorem 4.2.4, and Note 4.6.10 in \cite{davies}.
  Every ultracontractive semigroup has a kernel for the killed semigroup on a
   bounded domain which can be represented as a series expansion of the eigenvalues
   and the eigenfunctions of $L$  (cf. \cite[Theorems 2.1.4 and 2.3.6]{davies} and
    \cite[Theorems 8.37 and 8.38]{gilbarg-trudinger} ). There exist eigenvalues
$0< \mu_1<\mu_2\leq \mu_3\cdots,$ such that $\mu_n\to\infty,$ as
$n\to\infty$, with the  corresponding complete orthonormal set (in
$L^2$) of eigenfunctions $\psi_n$ of the operator $L$
satisfying

\begin{equation}\label{eigen-eigen}
 L \psi_n(x)=-\mu_n \psi_n(x), \ x\in D:\  \\
\psi_n |_{\partial D}=0.
\end{equation}

 In this case,
$$
p_D(t,x,y)=\sum_{n=1}^{\infty}e^{-\mu_n t}\psi_n(x)\psi_n(y)
$$
is the heat kernel of the killed semigroup $T_D$. The series
converges absolutely and  uniformly on $[t_0,\infty)\times D\times
D$ for all $t_0>0$.

Denote the Laplace transform $ t\rightarrow s $ of $u(t,x)$ by
\begin{equation*}\begin{split}
\tilde{u}(s,x)&=\int_{0}^{\infty}e^{-s t}u(t,x)dt .
\end{split}\end{equation*}
Since we are working on a bounded domain, the Fourier transform
methods in \cite{Zsolution} are not useful.  Instead, we will
employ Hilbert space methods.  Hence, given a complete orthonormal
basis $\{\psi_n(x)\}$ on $L^2(D)$, we will call
\begin{equation*}
\bar{u}(t,n)=\int_{D}\psi_n(x)u(t,x)dx,
\end{equation*}
 and
 \begin{equation}\begin{split} \label{eqn2.5v3}
\hat{u}(s,n)&=\int_{D}\psi_n(x)\int_{0}^{\infty}e^{-s t} u(t,x)dtdx   \\
&= \int_{D} \psi_{n}(x) \tilde{u}(s, x) dx   \\
& = \int_{0}^{\infty}e^{-s t} \bar{u}(t,x)dt~~\mbox{(when Fubini's
condition holds)}
 \end{split}\end{equation}
respectively the $\psi_n$ and the $\psi_n$-Laplace transforms.
Since $\{\psi_n\}$ is a complete orthonormal basis for $L^2(D)$,
we can invert the $\psi_n$-transform to obtain
\[u(t,x)=\sum_n \bar{u}(t,n) \psi_n(x)\]
for any $t>0$, where the above series converges in the $L^2$ sense
(e.g., see \cite[Proposition 10.8.27]{Royden}).

 Suppose $D$  satisfies the uniform exterior cone condition. Let $\{X_t\}$ be a Markov process
 in $\rd$ with generator $L$, and $f$
be continuous on $\bar D$. Then the semigroup
\begin{equation}\begin{split}\label{TDdef}
T_D(t)f(x)&=E_x[f(X_t)I(\tau_D(X)>t)]=\int_D p_D(t,x,y)f(y)dy=\sum_{n=1}^{\infty}e^{-\mu_n t}\psi_n(x)\bar{f}(n)
\end{split}\end{equation}
 solves the Dirichlet initial-boundary value problem in $D$:
\begin{eqnarray}
\frac{\partial u(t,x)}{\partial t}&=& L_D u(t,x),\ \ x\in D, \ t>0, \nonumber\\
u(t,x)&=&0,\ \ x\in \partial D,\nonumber\\
u(0,x)&=&f(x), \ \ x\in D.\nonumber
\end{eqnarray}
See \cite[Theorem 6.3, page 177]{bass} or \cite[Theorem 2.1.4]{davies}.




\begin{remark}
Let $L^{\infty}(D)= \{f: ||f||_{\infty}<\infty\}$, where
$||f||_{\infty} = ess~ sup ~{|f|}$. The eigenfunctions belong to
$L^\infty(D)\cap C^\alpha(D)$ for some $\alpha>0$,
 by \cite[Theorems 8.15 and 8.24]{gilbarg-trudinger}. If $D$ satisfies
 the
 {\em uniform exterior cone condition}, then all the eigenfunctions belong to $C^\alpha(\bar D)$
  by \cite[Theorem 8.29]{gilbarg-trudinger}. If $a_{ij}\in C^\alpha(\bar D)$ and  $\partial D \in C^{1,\alpha}$,
   then  the eigenfunctions belong to $C^{1,\alpha}(\bar D)$ by \cite[Corollary
8.36]{gilbarg-trudinger}.
  If $a_{ij} \in C^{\infty}(D)$, then each eigenfunction of $L$ is in $C^\infty(D)$ by
  \cite[Corollary 8.11]{gilbarg-trudinger}.
If $a_{ij} \in C^{\infty}(\bar{D})$ and $\partial D\in C^\infty,$
then each eigenfunction of $L$ is in $C^\infty(\bar{D})$ by
\cite[Theorem 8.13]{gilbarg-trudinger}.
\end{remark}


In the case $L=\Delta$, the Laplacian, the corresponding Markov
process is a  Brownian motion. We denote the eigenvalues and the
eigenfunctions of $\Delta$ on $D$, with Dirichlet boundary
conditions,  by $\{\lambda_n, \phi_n\}_{n=1}^\infty$, where
$\phi_n\in C^{\infty}(D)$.



\section{Distributed order fractional derivatives}\label{sec3}

Fractional derivatives in time are useful for physical models that involve sticking or trapping
 \cite{Zsolution}.  They are closely connected to random walk models with long waiting times between
 particle jumps \cite{limitCTRW}.  The fractional derivatives are essentially convolutions with a power law.
  Various forms of the fractional derivative can be defined, depending on the domain of the power law kernel,
  and the way boundary points are handled \cite{MillerRoss, Samko}.  The Caputo fractional derivative
   \cite{Caputo} is defined for $0<\beta<1$ as
\begin{equation}\label{CaputoDef}
\frac{\partial^\beta u(t,x)}{\partial
t^\beta}=\frac{1}{\Gamma(1-\beta)}\int_0^t \frac{\partial
u(r,x)}{\partial r}\frac{dr}{(t-r)^\beta} .
\end{equation}
Its Laplace transform
\begin{equation}\label{CaputoLT}
\int_0^\infty e^{-st} \frac{\partial^\beta u(t,x)}{\partial
t^\beta}\,ds=s^\beta \tilde u(s,x)-s^{\beta-1} u(0,x)
\end{equation}
incorporates the initial value in the same way as the first
derivative.  The Caputo derivative is useful for solving
differential equations that involve a fractional time
derivative \cite{GorenfloSurvey,Podlubny}, because it naturally incorporates initial values.

Fractional time derivatives emerge in anomalous diffusion models,
when particles wait a long time between jumps.
 In the standard model, called a continuous time random walk (CTRW), a particle waits a random time $J_n>0$ and
  then takes a step of random size $Y_n$.  For the purposes of this paper, we may assume that the two sequences
  of i.i.d.\ random variables $(J_n)$ and $(Y_n)$ are independent.  This is called an uncoupled CTRW.
  The particle
   arrives at location $X(n)=Y_1+\cdots+Y_n$ at time $T(n)=J_1+\cdots+J_n$.   Since
    $N_t=\max\{n\geq 0:T(n)\leq t\}$ is the number of jumps by time $t>0$, the particle location at
    time $t$ is $X(N_t)$.  If $EY_n=0$ and $E[Y_n^2]<\infty$ then, as the time scale $c\to\infty$, the
    random walk of particle jumps has a scaling limit $c^{-1/2}X([ct])\Rightarrow B(t)$, a standard Brownian
    motion.  If
$P(J_n>t)\sim ct^{-\beta}$ for some $0<\beta<1$ and $c>0$, then
the scaling limit $c^{-1/\beta}T([ct])\Rightarrow W_t$ is a
strictly increasing stable L\'evy process with index $\beta$,
sometimes called a stable subordinator.  The jump times $T(n)$ and
the number of jumps $N_t$ are inverses $\{N_{t}\geq n\}=\{T(n)\leq
t\}$, and it follows that the scaling limits are also inverses
 \cite[Theorem 3.2]{limitCTRW}:
$c^{-\beta}N_{ct}\Rightarrow E_t$, where
\begin{equation}\label{Etdef}
E_t=\inf\{\tau:W_\tau> t\} ,
\end{equation}
so that $\{E_t\leq \tau\}=\{W_\tau\geq t\}$.  A continuous mapping argument \cite[Theorem 4.2]{limitCTRW} yields
the CTRW scaling limit:  Heuristically, since $N_{ct}\approx c^{\beta}E_t$, we have
$c^{-\beta/2}X(N_{[ct]})\approx (c^\beta)^{-1/2}X(c^\beta
E_t)\approx B(E_t)$, a time-changed Brownian motion.  The density $u(t,x)$ of the process $B(E_t)$ solves a
fractional Cauchy problem
\[\frac{\partial^\beta u(t,x)}{\partial t^\beta}=a\frac{\partial^2 u(t,x)}{\partial x^2}\]
 for some $a>0$, where the order of the fractional derivative equals the index of the stable subordinator.
 Roughly speaking, if the probability of waiting longer than time $t>0$ between jumps falls off like
 $t^{-\beta}$, then the limiting particle density solves a diffusion equation that involves a fractional
  time derivative of the same order $\beta$.  Hence, the fractional derivatives in time model sticking or
  trapping of particles for long (power-law distributed) periods of time.

 A more flexible model for diffusion processes can be obtained by considering a sequence of CTRW.  At each scale
  $c>0$, we are given i.i.d.\ waiting times $(J_n^{c})$ and i.i.d.\ jumps $(Y_n^{c})$.  Assume the waiting times
   and jumps form triangular arrays whose row sums converge in distribution. Letting
   $X^{c}(n)=Y_1^{c}+\cdots+Y_n^{c}$ and $T^{c}(n)=J_1^{c}+\cdots+J_n^{c}$, we require that
$X^{c}(cu)\Rightarrow A(t)$ and $T^{c}(cu)\Rightarrow W_t$ as $c\to\infty$, where the limits $A(t)$ and $W_t$ are
 independent L\'evy processes.
Letting $N^{c}_t=\max\{n\geq 0:T^{c}(n)\leq t\}$, the CTRW scaling
limit $X^{c}(N_t^{c})\Rightarrow A(E_t)$
 \cite[Theorem 2.1]{M-S-triangular}.  A power-law mixture model for waiting times was proposed
  in \cite{M-S-ultra}:
   Take an i.i.d.\ sequence of mixing variables $(B_i)$ with $0<B_i<1$ and assume
   $P\{J_i^{c}>u|B_i=\beta\}=c^{-1}u^{-\beta}$ for $u\geq c^{-1/\beta}$, so that the waiting times are power
    laws conditional on the mixing variables.  The waiting time process $T^{c}(cu)\Rightarrow W_t$ a
    nondecreasing L\'evy process, or subordinator, with ${\mathbb
E}[e^{-s W_t}]=e^{-t\psi_W(s)}$  and Laplace exponent
\begin{equation}\label{phiWdef}
\psi_W(s)=\int_0^\infty(e^{-s x}-1)\phi_W(dx) .
\end{equation}
The L\'evy measure
\begin{equation}\label{psiWdef}
\phi_W(t,\infty)=\int_0^1 t^{-\beta}\mu(d\beta),
\end{equation}
where $\mu$ is the distribution of the mixing variable \cite[Theorem 3.4 and Remark 5.1]{M-S-ultra}.  A
computation \cite[Eq.\ (3.18)]{M-S-ultra} using
$\int_0^\infty(1-e^{-s t})\beta
t^{-\beta-1}dt=\Gamma(1-\beta)s^\beta$ shows that
\begin{equation}\begin{split}\label{psiW}
\psi_W(s)
&= \int_0^1  s^\beta \Gamma(1-\beta) \mu(d\beta) .
\end{split}\end{equation}
 Then $c^{-1}N^{c}_t\Rightarrow E_t$, the inverse subordinator \cite[Theorem 3.10]{M-S-ultra}.  The general
 infinitely divisible L\'evy process limit $A(t)$ forms a strongly continuous convolution semigroup with
 generator $L$ (e.g., see \cite{ABHN}) and the corresponding CTRW scaling limit $A(E_t)$ is the stochastic
 solution to the distributed order-fractional Cauchy problem \cite[Eq.\
 (5.12)]{M-S-ultra}defined by
\begin{equation}\label{DOFCPdef}
\D^{(\nu)} u(t,x)=L u(t,x),
\end{equation}
where the distributed order fractional derivative
\begin{equation}\label{DOFDdef}
\D^{(\nu)}u(t,x):=\int_0^1\frac{\partial^\beta u(t,x)}{\partial
t^\beta} \nu(d\beta) ,\quad \text{ $\nu(d\beta)=\Gamma(1-\beta)\mu(d\beta)$.}
\end{equation}
To ensure that $\D^{(\nu)}$ is well-defined, we impose the condition
\begin{equation}\label{finite-mu-bound}
\int_0^1  \frac 1{1-\beta}\, \mu(d\beta)<\infty
\end{equation}
as in  \cite[Eq.\ (3.3)]{M-S-ultra}.  Since $\Gamma(x)\sim 1/x$, as $x\to 0+$, this ensures that $\nu(d\beta)$ is a finite measure on $(0,1)$.

Using triangular array limits for CTRW allows a more flexible limit model.  For example, suppose $Y_i$ are
i.i.d.\ with mean $\mu$ and variance $\sigma^2$, and let $Y_i^{c}=c^{-1}\mu+c^{-1/2}(T_i-\mu)$ so that
 $X^{c}(cu)\Rightarrow A(t)$ a Brownian motion with drift.  Then the density $u(t,x)$ of the CTRW scaling
 limit $A(E_t)$ solves \eqref{DOFCPdef} with  \[L=-v\frac{\partial }{\partial x}+a\frac{\partial^2 }
 {\partial x^2},\]
for some $a>0$.  A triangular array of jumps with two spatial scales, one for the mean jump and another for the
 deviation from the mean, is necessary to get a drift term in the limit.


Since $\phi_W(0,\infty)=\infty$ in \eqref{phiWdef}, Theorem 3.1
in \cite{M-S-triangular} implies that the inverse subordinator $E_t$ has
density
\[
g(t,x)=\int_0^t \phi_W(t-y,\infty) P_{W(x)}(dy) .
\]
This same condition  ensures also that $E_t$ is almost surely
continuous, since $W_t$ jumps in every interval, and hence is
strictly increasing. Further, it follows from the definition
\eqref{Etdef} that $E_t$ is monotone nondecreasing.





The following lemma shows that $ h(t, \lambda)=\E [e^{-\lambda
E_t}]$ is an eigenfunction of the
 distributed-order fractional derivative $\D^{(\nu)}$.

\begin{lemma}\label{eigenvalue-problem}
For any $\lambda>0$, $h(t, \lambda)=\int_0^\infty
e^{-\lambda x}g(t,x)\,dx=\E [e^{-\lambda E_t}]$ satisfies
\begin{equation}\label{dist-order-density-pde}
\D^{(\nu)}h(t,\lambda)=-\lambda h(t, \lambda); \ \ h(0, \lambda)=1.
\end{equation}
\end{lemma}

\begin{proof} First note that $h(0, \lambda) =\mathbb{E}(1) = 1$.
Using \eqref{CaputoLT} and \eqref{DOFDdef}, compute the Laplace
transform of $\D^{(\nu)}h(t,\lambda)$
as
\begin{equation}\label{eqn3.14n}\begin{split}
\int_{0}^{\infty} e^{-s t}\,\D^{(\nu)} h(t,\lambda)dt &= \int_0^\infty e^{-st} \int_0^1
\frac{\partial^\beta h(t,\lambda)}{\partial
t^\beta} \nu(d\beta)dt\\
&=\int_0^1 \int_0^\infty e^{-st}  \frac{\partial^\beta h(t,\lambda)}{\partial
t^\beta}dt \nu(d\beta)\\
&=\int_0^1 (s^\beta\tilde{h}(s, \lambda)-s^{\beta-1}) \nu(d\beta)\\
&=\left(\tilde{h}(s, \lambda) -\frac{1}{s}\right)\psi_W(s),
\end{split}
\end{equation}
by applying a Fubini argument which holds because $\psi_W(s)<\infty$.

The Laplace transform of $g(t,x)$ is given by \cite[Eq.\ (3.13)]{M-S-triangular}:
\begin{equation}\label{laplace-distributed-inverse}
\tilde g(s,x)=\int_0^\infty e^{-s t}g(t,x)dt=\frac{1}{s}\psi_W(s)e^{-x\psi_W(s)}.
\end{equation}
Then the double Laplace transform
\begin{eqnarray}
\tilde h(s,\lambda):=\int_0^\infty e^{-s t}h(t,\lambda)dt
&=& \int_0^\infty e^{-\lambda t}\bigg(\int_0^\infty
e^{-\lambda x}g(t,x)dx\bigg)dt\nonumber\\
&=& \int_0^\infty e^{-\lambda x}\bigg(\int_0^\infty e^{-s t} g(t,x)dt\bigg)dx\nonumber\\
&=& \frac{\psi_W(s)}{s} \int_0^\infty e^{-(\lambda + \psi_W(s))x}dx\label{double-laplace}\\
&=& \frac{\psi_W(s)}{s(\lambda+\psi_W(s))} \label{eqn3.20v3} .
\end{eqnarray}
That is,
$\tilde{h}(s, \lambda)$ satisfies
\begin{equation}\label{eqn3.16n}
\lambda\tilde{h}(s,\lambda)=
\Bigl(\frac{1}{s}-\tilde{h}(s,\lambda)\Bigr)\psi_W(s).
\end{equation}
Since  $E_t$ has continuous paths, the dominated convergence
theorem implies that
 $t\to\E [e^{-\lambda E(t)}]=h(t, \lambda)$ is a continuous function.
Then \eqref{dist-order-density-pde} follows from (\ref{eqn3.14n}),
(\ref{eqn3.16n}) and the uniqueness of the Laplace transform.
\end{proof}

\begin{lemma}\label{derivative-bound}
Suppose that $\mu(d\beta)=p(\beta)d\beta$,
  the function $\beta\mapsto \Gamma(1-\beta)p(\beta)$ is in $C^1[0,1]$, $supp (\mu)=[\beta_0,\beta_1]\subset (0,1)$
   and $\mu(\beta_1)\neq 0$.
Suppose also that 
\begin{equation} \label{eqn3.2v3}
 C( \beta_0, \beta_1,p)= \int_{\beta_0}^{\beta_1}\sin
(\beta\pi)\Gamma(1-\beta)p(\beta)d\beta>0.
\end{equation}
Then $|\partial_t h(t, \lambda)|\leq \lambda k(t)$, where
\begin{equation}\label{time-bound-derivative}
k(t)=[C(\beta_0,\beta_1,p)\pi]^{-1}[\Gamma(1-\beta_1)t^{\beta_1-1}+\Gamma(1-\beta_0)t^{\beta_0-1}].
\end{equation}
In this case, $h(t,\lambda)$ is a classical solution to \eqref{dist-order-density-pde}.
\end{lemma}

\begin{proof}

Using  (2.19) in Kochubei \cite{koch3}, which follows from inverting the
Laplace transform in \eqref{double-laplace} of $h(t,\lambda)$, we
have
\begin{equation}
h(t,\lambda)=\frac{-\lambda}{\pi}\int_0^\infty r^{-1}e^{-tr}\Phi(r,1)dr
\end{equation}
where
$$
\Phi(r,1)=\frac{\int_0^1r^\beta\sin
(\beta\pi)\Gamma(1-\beta)p(\beta)d\beta}{[\int_0^1r^\beta \cos
(\beta\pi)\Gamma(1-\beta)p(\beta)d\beta+\lambda]^2+[\int_0^1r^\beta\sin
(\beta\pi) \Gamma(1-\beta)p(\beta)d\beta]^2}.
$$
First we show that $| \partial_t h(t, \lambda)| < \infty$. Note
that
\begin{equation}
\begin{split}
&|\partial_t h(\lambda,t)|=\left|\frac{-\lambda}{\pi}\int_0^\infty r^{-1}[\partial_te^{-tr}]\Phi(r,1)dr \right|\\
&=\frac{\lambda}{\pi}\int_0^\infty e^{-tr}\Phi(r,1)dr\\
&=\frac{\lambda}{\pi}\int_0^\infty \frac{ e^{-tr}\int_0^1r^\beta\sin (\beta\pi)
\Gamma(1-\beta)p(\beta)d\beta}{[\int_0^1r^\beta\cos (\beta\pi)\Gamma(1-\beta)
p(\beta)d\beta+\lambda]^2+[\int_0^1r^\beta\sin (\beta\pi)\Gamma(1-\beta)p(\beta)d\beta]^2}dr\\
&\leq \lambda \pi^{-1} \int_0^\infty\frac{e^{-tr}dr}{\int_0^1r^\beta\sin (\beta\pi)\Gamma(1-\beta)p(\beta)d\beta}\\
&=\lambda l(t), ~\mbox{(say)},
\end{split}
\end{equation}
 where $l(t)$ is a function of $t$ only.  In the case of a simple fractional derivative this $l(t)$ is given
 by $Ct^{\beta-1}$.


 Now,
\begin{equation}
\int_{0}^{1}\sin (\beta\pi)\Gamma(1-\beta)p(\beta)d\beta\geq
\int_{\beta_0}^{\beta_1} \sin
(\beta\pi)\Gamma(1-\beta)p(\beta)d\beta= C(\beta_0, \beta_1,p)>0,
\end{equation}
by assumption (\ref{eqn3.2v3}).


For $r>1$, and $\beta_0\leq \beta\leq \beta_1\leq 1,$ we have
$r^{\beta_0}\leq r^{\beta}\leq r^{\beta_1}$ and so
\begin{equation}
\begin{split}
\int_{\beta_0}^{\beta_1}r^\beta\sin (\beta\pi)\Gamma(1-\beta)p(\beta)d\beta & \geq
\int_{\beta_0}^{\beta_1}r^{\beta_0}\sin (\beta\pi)\Gamma(1-\beta)p(\beta)d\beta\\
&=r^{\beta_0}C(\beta_0,\beta_1,p).
\end{split}
\end{equation}
For $0<r\leq 1$, and $\beta_0\leq \beta\leq \beta_1\leq 1,$ we
have $r^{\beta_0}\geq r^{\beta}\geq r^{\beta_1}$ and so
\begin{equation}
\begin{split}
\int_{\beta_0}^{\beta_1}r^\beta\sin (\beta\pi)\Gamma(1-\beta)p(\beta)d\beta & \geq
\int_{\beta_0}^{\beta_1}r^{\beta_1}\sin (\beta\pi)\Gamma(1-\beta)p(\beta)d\beta\\
&=r^{\beta_1}C(\beta_0,\beta_1,p).
\end{split}
\end{equation}
Using the above facts, we obtain
\begin{equation*}
\begin{split}
l(t)&=\pi^{-1} \int_0^\infty\frac{e^{-tr}dr}{\int_0^1r^\beta\sin (\beta\pi)\Gamma(1-\beta)p(\beta)d\beta}\\
&=\pi^{-1}\bigg[\int_0^1\frac{e^{-tr}dr}{\int_0^1r^\beta\sin (\beta\pi)\Gamma(1-\beta)p(\beta)d\beta}
+\int_1^\infty\frac{e^{-tr}dr}{\int_0^1r^\beta\sin (\beta\pi)\Gamma(1-\beta)p(\beta)d\beta}\bigg]\\
&\leq \pi^{-1}\bigg[\int_0^1\frac{e^{-tr}dr}{\int_{\beta_0}^{\beta_1}r^\beta\sin (\beta\pi)
\Gamma(1-\beta)p(\beta)d\beta}+\int_1^\infty\frac{e^{-tr}dr}{\int_{\beta_0}^{\beta_1}r^\beta
\sin (\beta\pi)\Gamma(1-\beta)p(\beta)d\beta}\bigg]\\
&\leq [C(\beta_0,\beta_1,p)\pi]^{-1}\bigg[\int_0^1r^{-\beta_1}e^{-tr}dr+\int_1^\infty r^{-\beta_0}e^{-tr}dr\bigg]\\
&\leq [C(\beta_0,\beta_1,p)\pi]^{-1}[\Gamma(1-\beta_1)t^{\beta_1-1}+\Gamma(1-\beta_0)t^{\beta_0-1}]=k(t)
\end{split}
\end{equation*}
and so $|\partial_t h(t, \lambda)|\leq \lambda k(t)$. Hence, it
follows from (\ref{DOFDdef}) that
\begin{equation*}\begin{split}
|\D^{(\nu)}h(t,\lambda)|&\leq  \left|\int_0^1\frac{\partial^\beta}{\partial t^\beta}h(t,\lambda) \Gamma(1-\beta)p(\beta)d\beta\right|\\
&\leq \int_0^1\frac{1}{\Gamma(1-\beta)}\int_0^t \left|\frac{\partial h(s,\lambda)}{\partial s}\right|\frac{ds}{(t-s)^\beta} \Gamma(1-\beta)p(\beta)d\beta \\
&\leq \int_0^1\frac{1}{\Gamma(1-\beta)}\int_0^t k(s)\frac{ds}{(t-s)^\beta} \Gamma(1-\beta)p(\beta)d\beta \\
&<\infty,
\end{split}\end{equation*}
using \eqref{finite-mu-bound} and the beta density formula. Thus,
the distributed-order derivative $\D^{(\nu)}h(t,\lambda)$ is
defined in the classical sense.
\end{proof}

\section{Distributed order fractional Cauchy problems}\label{sec4}
Fractional Cauchy problems replace the usual first-order time
derivative with its fractional analogue.  In this section, we
prove classical (strong) solutions to distributed-order fractional
Cauchy problems $\D^{(\nu)} u=L u$ on bounded domains $D \subset
\rd$.  We  give also an explicit solution formula, based on the
solution of the corresponding Cauchy problem.  Our methods are
inspired by the approach in \cite{Zsolution}, where Laplace
transforms are used to handle the fractional time derivative, and
spatial derivative operators (or more generally,
pseudo-differential operators) are treated with Fourier
transforms.  In the present paper, we use an eigenfunction
expansion in place of Fourier transforms, since we are operating
on a bounded domain.  Our first result, Theorem \ref{laplace-pde},
is focused on a distributed-order fractional diffusion with
$L=\Delta$, and we lay out all the details of the argument in the
most familiar setting.  Then, in Theorem \ref{FC-PDE-Lx}, we use
separation of variables to extend this approach to uniformly
elliptic generators $L$. In the process, we explicate the
stochastic solutions in terms of killed Markov processes.



Let $D_\infty=(0,\infty)\times D$ and  define
\begin{eqnarray}
\mathcal{H}_\Delta(D_\infty)&\equiv & \left\{u:D_\infty\to \rr :
\Delta u\in C(D_\infty),\right.\nonumber\\
& &\left. \left|\partial_t u(t,x)\right|\leq k(t)g(x),\ g\in
L^\infty(D), \ t>0 \right\},\nonumber
\end{eqnarray}
where $k(t)$ is defined by \eqref{time-bound-derivative}.




We will write $u\in C^k(\bar D)$ to mean that for each fixed
$t>0$, $u(t,\cdot)\in C^k(\bar D)$, and
 $u\in C_b^k(\bar D_\infty)$ to mean that $u\in C^k(\bar D_\infty)$ and is bounded.
Let $\tau_D(X)= \inf\{ t\geq 0: X(t)\notin D\}$ denote the first
exit time of the stochastic process $X= \{X(t)\}$.

 \begin{theorem}\label{laplace-pde}
Let $D$ be a bounded domain with $\partial D \in
 C^{1,\alpha}$ for some $0<\alpha<1$, and
  $T_D(t)$  be the killed semigroup of Brownian motion  $\{X(t)\}$ on $D$.  Let $E_t$ be the inverse \eqref{Etdef}
of the subordinator $W_t$, independent of $\{X(t)\}$, with L\'evy
measure \eqref{psiWdef}. Suppose that
 $\mu(d\beta)=p(\beta)d\beta$, as in Lemma \ref{derivative-bound},
 and $\D^{(\nu)}$ is the distributed-order fractional derivative  defined by
 \eqref{DOFDdef}.
Then, for any $f\in D(\Delta_D)\cap C^1(\bar D)\cap C^2(D)$ for
which the
  eigenfunction expansion (of $\Delta f$) with respect to the complete orthonormal basis
   $\{\phi_n:\ n\in \N \}$ converges uniformly and absolutely,
the unique (classical) solution of the distributed order fractional
Cauchy problem
\begin{eqnarray}
\D^{(\nu)}u(t,x) &=&
\Delta u(t,x);  \  \ x\in D, \ t>0\label{frac-derivative-bounded-d}\\
u(t,x)&=&0, \ x\in \partial D, \ t>0, \nonumber\\
u(0,x)& =& f(x), \ x\in D,\nonumber
\end{eqnarray}
for  $u  \in
\mathcal{H}_\Delta(D_\infty)\cap C_b(\bar D_\infty) \cap C^1(\bar D)$,
is given by
\begin{eqnarray}
 u(t,x)&=&E_{x}[f(X(E_{t}))I( \tau_D(X)> E_t)]\nonumber\\
&=& \int_{0}^{\infty}T_D(l)f(x)g(t,l)dl \label{solution-laplacian-frac-cauchy}\\
&=&\sum_{n=1}^{\infty}\bar{f}(n)\phi_n(x)h(t,\lambda_n),
\end{eqnarray}
where $h(t,\lambda)=\E (e^{-\lambda E_t})=\int_0^\infty
e^{-\lambda x}g(t,x)\,dx$ is the Laplace transform of $E_t$.
\end{theorem}

\begin{proof}
 Assume that $u(t,x)$ solves
\eqref{frac-derivative-bounded-d}. Using Green's second identity,
we obtain
\begin{equation*}
\int_D [u\Delta \phi_n-\phi_n\Delta u]dx=\int_{\partial
D}\left[u\frac{\partial \phi_n}{\partial
\theta}-\phi_n\frac{\partial u}{\partial \theta}\right]ds=0,
\end{equation*}
since $u|_{\partial D}=0=\phi_n|_{\partial D}$ and $u,\phi_n\in
C^1(\bar{D})$. Hence, the $\phi_n$-transform of $\Delta u$ is
$$\int_D \phi_n(x) \Delta u(t,x)  dx=\int_D  u(t,x)\Delta \phi_n(x) dx=
-\lambda_n\int_D  u(t,x) \phi_n(x) dx = -\lambda_n\bar{u}(t,n),$$
as $\phi_n$ is the eigenfunction of the Laplacian corresponding to
 eigenvalue $\lambda_n$.













\noindent Next we need to show that the $\phi_n$ transform
commutes with $\D^{(\nu)}$.
We need to show that  we can interchange derivatives and integrals
as follows. Observe that
\begin{equation}\label{real-fubini}
\begin{split}
&\int_D \phi_n(x) \D^{(\nu)}u(t,x) dx\\
&=\int_D \phi_n(x) \int_0^1\frac{\partial^\beta}{\partial t^\beta}u(t,x) \Gamma(1-\beta)p(\beta)d\beta dx\\
&=\int_D \phi_n(x) \int_0^1\frac{1}{\Gamma(1-\beta)}\int_0^t \frac{\partial u(s,x)}{\partial s}\frac{ds}{(t-s)^\beta} \Gamma(1-\beta)p(\beta)d\beta dx\\
&=\int_D \phi_n(x) \int_0^1 \int_0^t \frac{\partial u(s,x)}{\partial s}\frac{ds}{(t-s)^\beta} p(\beta)d\beta dx\\
&=\int_0^1 \int_0^t \left(\int_D \phi_n(x) \frac{\partial }{\partial s}u(s,x)dx\right) \frac{ds}{(t-s)^\beta} p(\beta)d\beta \quad \text{(by Fubini, see below)}\\
&=\int_0^1 \int_0^t \frac{\partial }{\partial s}\left(\int_D \phi_n(x) u(s,x)dx\right) \frac{ds}{(t-s)^\beta} p(\beta)d\beta \\
&=\int_0^1 \frac{1}{\Gamma(1-\beta)}\int_0^t \frac{\partial }{\partial s}\bar u(s,n) \frac{ds}{(t-s)^\beta} \Gamma(1-\beta)p(\beta)d\beta \\
&=\D^{(\nu)} \bar u(s,n).
 \end{split}\end{equation}

The Fubini-Tonelli argument for the interchange of order of
integration
 in \eqref{real-fubini} can be justified as follows:
\begin{equation*}
\begin{split}
&\bigg|\int_D \phi_n(x) \D^{(\nu)}u(t,x) dx\bigg|\\
&=\bigg|\int_D \phi_n(x) \int_0^1 \int_0^t \frac{\partial
u(s,x)}{\partial s}\frac{ds}{(t-s)^\beta}
 p(\beta)d\beta dx\bigg|\\
&\leq \int_D |\phi_n(x)| \int_0^1 \int_0^t \bigg|\frac{\partial
u(s,x)}{\partial s}\bigg|\frac{ds}
{(t-s)^\beta} p(\beta)d\beta dx\\
&\leq  \int_D |\phi_n(x)||g(x)|dx \int_0^1 \int_0^t k(s)\frac{ds}{(t-s)^\beta} p(\beta)d\beta \\
&\leq \sqrt{|D|}||\phi_n||_{L^2(D)}||g||_{L^\infty}\int_0^1
\int_0^t [C(\beta_0,\beta_1,p)\pi]^{-1}
[\Gamma(1-\beta_1)s^{\beta_1-1}+\Gamma(1-\beta_0)s^{\beta_0-1}]\\
&\ \ \ \ \ \ \ \ \ \  \ \ \ \ \ \ \ \ \ \ \ \ \ \ \ \ \ \ \ \ \ \ \ \ \ \ \ \ \ \times \frac{ds}{(t-s)^\beta} p(\beta)d\beta,
\end{split}
\end{equation*}
using (\ref{time-bound-derivative}). Further, using the property of beta
density, for $0<\gamma, \eta <1 $,
\[\int_0^t\frac{1}{(t-s)^\gamma}s^{\eta-1}ds=t^{\eta-\gamma}\int_0^1(1-u)^{(1-\gamma)-1}u^{\eta-1}du
= B(1-\gamma, \eta)t^{\eta-\gamma}, \] where $B(a, b)$ denotes the
usual beta function. Thus,
\begin{equation*}
\begin{split}
&\bigg|\int_D \phi_n(x) \D^{(\nu)}u(t,x) dx\bigg|\\
&\leq
\sqrt{|D|}||\phi_n||_{L^2(D)}||g||_{L^\infty}[C(\beta_0,\beta_1,p)\pi]^{-1}\bigg[\int_0^1
\int_0^t \Gamma(1-\beta_1)s^{\beta_1-1}\frac{ds}{(t-s)^\beta} p(\beta)d\beta\\
&+\int_0^1 \int_0^t\Gamma(1-\beta_0)s^{\beta_0-1}\frac{ds}{(t-s)^\beta} p(\beta)d\beta\bigg]\\
&=\sqrt{|D|}||\phi_n||_{L^2(D)}||g||_{L^\infty}[C(\beta_0,\beta_1,p)\pi]^{-1}\bigg[\Gamma(1-\beta_1)
\int_0^1t^{\beta_1-\beta}  B(1-\beta, \beta_1) p(\beta)d\beta \\
&+\Gamma(1-\beta_0)\int_0^1t^{\beta_0-\beta} B(1-\beta, \beta_0)p(\beta)d\beta\bigg]\\
&<\infty,
\end{split}
\end{equation*}
which justifies the use of  Fubini-Tonelli theorem in
\eqref{real-fubini}.

Thus, applying the  $\phi_n$-transforms to
(\ref{frac-derivative-bounded-d}), we get
\begin{equation}\label{phi-trans}
\D^{(\nu)}\bar{u}(t,n)=
 -\lambda_n
\bar{u}(t,n).
\end{equation}

\noindent Since $u$ is uniformly continuous on $C ([0,\epsilon])
\times \bar D)$, it is uniformly bounded on $[0,\epsilon] \times
\bar D $. Thus,  by the dominated convergence theorem, we have
$\lim_{t\to 0}\int_D u(t,x)\phi_n(x)dx=\bar f(n)$.  Hence,
$\bar{u}(0,n)=\bar f(n)$.  A similar argument shows that $t\mapsto
\bar u(t,n)$ is a continuous function of $t\in [0,\infty)$ for
every $n$.   Then, taking Laplace transforms on both sides of
\eqref{phi-trans}, we get
\begin{equation}\label{laplace-phi-trans}
\int_0^1
(s^\beta\hat{u}(s,n)-s^{\beta-1}\bar{u}(0,n))\Gamma(1-\beta)p(\beta)d\beta
= -\lambda_n \hat{u}(s,n)
\end{equation}
which leads to
\begin{equation}\label{laplace-phi-trans-1}
\hat{u}(s,n)=\frac{\bar{f}(n)\int_0^1
s^{\beta-1}\Gamma(1-\beta)p(\beta)d\beta}{\int_0^1s^{\beta}\Gamma(1-\beta)p(\beta)d\beta+\lambda_n}.
\end{equation}
Recall $\int_0^1s^{\beta}\Gamma(1-\beta)p(\beta)d\beta = \psi
_W(s)$.
Then, from (\ref{laplace-phi-trans-1}),
\begin{eqnarray} \label{eqn3.99}
\hat{u}(s,n)&=&\frac{\bar{f}(n)\psi_W(s)}{s(\psi_W(s)+\lambda_n)}\nonumber\\
&=&\frac{1}{s}\bar{f}(n)\psi_W(s)
\int_0^{\infty}e^{-(\psi_W(s)+\lambda_n)l}dl\nonumber \\
&=&\int_0^{\infty}e^{-\lambda_n
l}\bar{f}(n)\frac{1}{s}\psi_W(s)e^{-l\psi_W(s)}dl,
\end{eqnarray}
using the property of the exponential density.


\noindent The $\phi_n$-transform of the killed semigroup
$T_D(l)f(x)=\sum_{m=1}^\infty e^{-\lambda_m l}\phi_m (x)\bar f
(m)$ from \eqref{TDdef} is found as follows. Since $\{\phi_n, n\in
\N \}$ is a complete orthonormal basis of $L^2(D),$ we get
\begin{equation}\label{MMMc}
\begin{split}
\overline{[T_D(l)f]}(n)&=\int_D \phi_n(x)T_D(l)f(x)dx\\
&=\int_D \phi_n(x)\int_D p_D(l,x,y) f(y)dydx\\
&=\int_D \phi_n(x)\int_D \sum_{m=1}^\infty e^{-\lambda_m l}\phi_m(x)\phi_m(y) f(y)dydx\\
&=\int_D \phi_n(x)\sum_{m=1}^\infty e^{-\lambda_m l}\phi_m(x)\int_D \phi_m(y) f(y)dydx\\
&=\int_D \phi_n(x)\sum_{m=1}^\infty e^{-\lambda_m l}\phi_m(x)\bar f(m)dx\\
&=\sum_{m=1}^\infty e^{-\lambda_m l}\bar f(m)\int_D \phi_n(x)\phi_m(x)dx\\
&=e^{-l\lambda_n}\bar{f}(n).
 \end{split}
 \end{equation}


Since $T_D(t)$ is a contraction semigroup on $L^2(D)$, $T_D(t)f\in
L^{2}(D)$ and hence Fubini-Tonelli applies.

\noindent By (3.20) in \cite{M-S-ultra},
we have
\begin{equation}\label{eqn3.11n}
\frac{1}{s}\psi_W(s)e^{-\psi_W(s)l}=\int_0^{\infty}e^{-s t}g(t,l)dt,
\end{equation}
where $g(t,l)$ is the smooth density of $E_t$.

\noindent Using the results  (\ref{MMMc}),  (\ref{eqn3.11n}) and
(\ref{eqn3.99}), we get
\begin{eqnarray*}
\int_{0}^{\infty} e^{-st} \bar{u}(t,n)dt&=&\hat{u}(s,n)=\int_0^{\infty}\overline{[T_D(l)f]}(n)\left[\int_0^{\infty}e^{-st}g(t,l)dt
\right]dl\nonumber\\
&=&\int_0^{\infty}e^{-st}\left[\int_0^{\infty}\overline{[T_D(l)f]}(n)g(t,l)dl
\right]dt
\end{eqnarray*}
using (\ref{eqn2.5v3}).

\noindent By the uniqueness of the  Laplace transform,
\begin{eqnarray} \label{eqn3.16v3}
\bar{u}(t,n)&=& \int_0^{\infty}\overline{[T_D(l)f]}(n)g(t,l)dl \nonumber \\
 &=& \bar{f}(n)\int_0^{\infty}e^{-\lambda_n l} g(t, l)dl ~~ (\mbox{
 using} (\ref{MMMc})) \nonumber \\
&=&\bar{f}(n) h(t, \lambda_n),
\end{eqnarray}
where $h(t, \lambda)=\int_0^{\infty}e^{-l\lambda}g(t,l)dl$ is the
Laplace transform of $E_t$.

\noindent Note that inverting the $\phi_n$-transform is equal to
multiplying both sides of the above equation by $\phi_n$ and then
summing up from $n=1$ to $\infty.$
Also, the unique  inverse  $ u(t,~ .) \in L^2(D)$ for each fixed
$t\geq 0$. Inverting the $\phi_n$-transform $ \bar{u}(t, n)$ in
(\ref{eqn3.16v3}), we get an $L^2$-convergent solution of
 (\ref{frac-derivative-bounded-d}) as
\begin{eqnarray}
u(t,x) &=& \sum_{n=1}^{\infty}\bar{u}(t,n) \phi_n(x) \nonumber \\
&=&\sum_{n=1}^{\infty}\bar{f}(n)\phi_n(x)h(\lambda_n,t)\label{mittag-series}
\end{eqnarray}
for each $t\geq 0$.

  In order to complete the proof, it will suffice to show that the series
  \eqref{mittag-series} converges pointwise, and satisfies all the conditions in
  \eqref{frac-derivative-bounded-d}.

\noindent {\bf  Step 1.} We begin showing that
\eqref{mittag-series} convergence uniformly in $t\in [0,\infty)$
in the $L^2$ sense. Define the sequence of functions
\begin{equation}\label{PS1}
u_N(t,x)=\sum_{n=1}^N \bar f(n)\phi_n(x)h(t,\lambda_n).
\end{equation}

\noindent Since $g(t,l)$ is the density of $E_t$, we have  $0<
h(t,\lambda_n)\leq 1$. Also, if  $0<\lambda\leq \eta,$ then
$h(t,\eta) \leq h(t,\lambda)$, showing that $ h(t, \lambda)$ is
nonincreasing in $\lambda.$


Since $f\in L^2(D)$, we can write $f(x)=\sum_{n=1}^\infty \bar
f(n)\phi_n(x)$,  and then the Parseval identity yields
$$\sum_{n=1}^\infty (\bar f(n))^2=||f||_{2,D}^2<\infty.$$
Then, given $\epsilon>0$, we can choose $n_0(\epsilon)$ such that
\begin{equation}\label{f-l2}
\sum_{n=n_0(\epsilon)}^\infty (\bar f(n))^2<\epsilon .
\end{equation}
For $N>M>n_0(\epsilon)$ and $t\geq 0,$
\begin{eqnarray}
||u_N(t,x)-u_M(t,x)||_{2,D}^2& \leq &\left\| \sum_{n=M}^N \bar
f(n)\phi_n(x)h(t,\lambda_n)\right\|_{2,D}^2\nonumber\\
&\leq &h(t,\lambda_{n_0})^2\sum_{n=n_0(\epsilon)}^\infty (\bar
f(n))^2\leq \sum_{n=n_0(\epsilon)}^\infty (\bar f(n))^2<\epsilon.
\end{eqnarray}
Thus, the series \eqref{mittag-series} converges in $L^2(D)$,
uniformly in $t\geq 0$.

\vskip10pt \noindent {\bf  Step 2.} Next we show that the initial
function is defined as the $L^2$ limit of $u(t,x)$ as $t\to 0,$
that is, we show that $t\to u(t, \cdot)\in C((0,\infty); L^2(D))$
and $u(t,\cdot)$ takes the initial datum $f$ in the sense of
$L^2(D)$, i.e.,
$$
||u(t, \cdot)-f||_{2,D}\to 0, \ \mathrm{as} \ t\to 0.
$$
Since $h(t, \lambda)$  is the Laplace transform of $E_t$, it is
completely monotone and non-increasing in $\lambda\geq 0$.
Hence,
$$u(t,x)-f(x)=\sum_{n=1}^\infty \bar f (n)(h(t,\lambda_n)-1)\phi_n(x).$$

Fix $\epsilon \in (0,1)$ and choose $ n_0=n_0(\epsilon)$ as in
(\ref{f-l2}). Then,
\begin{eqnarray}
||u(t,\cdot)-f||_{2,D}^2&=&\sum_{n=1}^\infty (\bar f (n))^2(h(t,\lambda_n)-1)^2\nonumber\\
& \leq & \sum_{n=1}^{n_0(\epsilon)} (\bar f (n))^2(h(t, \lambda_n)-1)^2\nonumber\\
& &+\sum_{n=n_0(\epsilon)+1}^{\infty} (\bar f (n))^2(h(t, \lambda_n)-1)^2\nonumber\\
&\leq & (1-h(t,\lambda_{n_0}))^2||f||_{2,D}+\epsilon\nonumber
\end{eqnarray}
and now  the claim follows, if $h(t,\lambda_n)\to 1,$ as $t\to 0$.

This follows because $E_t\Rightarrow E_0$ in distribution as $t\to
0+$ and hence the Laplace transforms converge.  To see that
$E_t\Rightarrow E_0$, use the fact that $\{E_t\leq x\}=\{W_x\geq
t\}$ which is (3.16) in \cite{M-S-ultra}.  Then for $x>0$ and
$t_n\downarrow 0,$
\[P(E_{t_n}\leq x)=P(W_x\geq t_n)=1-P(W_x\leq t_n)\to 1-P(W_x\leq
0)=1,\] since $W_x$ has a density. As $E_t$ also has a density on
$[0,\infty)$ we see that $P(E_{t_n}\leq x)=0$, if $x< 0.$ Thus,
\[P(E_{t_n}\leq x)\to P(E_0\leq x)=I(x\geq 0)\]
at all continuity points of the limit.  Note that $E_0=0$ almost
surely, since $W_x>0$ almost surely for all $x>0$.

The continuity of $t\mapsto u(t,\cdot)$ in $L^2(D)$ at every point
$t\in (0,\infty)$ can be proved in a similar fashion.

\vskip10pt \noindent {\bf  Step 3.}  A decay estimate for the
solution $u(t,x)$ is obtained as follows. Using Parseval's
identity, the fact that $\lambda_n$ is increasing in $n$, and the
fact that $h(t,\lambda_n)$ is non-increasing for $n\geq 1$, we get
$$
||u(t,\cdot)||_{2,D} \leq h(t,\lambda_1)||f||_{2,D}.
$$

\vskip10pt
 \noindent {\bf  Step 4.}  We next show that the
series (\ref{mittag-series}) defining $u(t,x)$ is the classical
solution to \eqref{frac-derivative-bounded-d}, by proving its
uniform and absolute convergence. We do this by showing that
\eqref{PS1} is a Cauchy sequence in $L^\infty (D)$ uniformly in
$t\geq 0.$

\noindent Applying the Green's second identity, we see that
$$
\overline{\Delta f}(n)=\int_D \Delta f(x)\phi_n (x)dx=-\lambda_n
\bar{f}(n).
$$
Hence,  $\Delta f = \sum_{n=1}^\infty -\lambda_n \phi_n(x)\bar f
(n)$  is absolutely and uniformly convergent by assumption. Let $\epsilon>0.$
Since $\lambda_n\to\infty$ as $n\to\infty$, there exists an
$n_0(\epsilon)$ so that for all $x\in D,$
\begin{equation}
\sum_{n=n_0(\epsilon)}^\infty |\bar f(n)||\phi_n(x)|\leq
\sum_{n=n_0(\epsilon)}^\infty \lambda _n |\bar f(n)||\phi_n(x)|
<\epsilon.
\end{equation}
This is possible since we have assumed that $\Delta f$ has an
expansion which is uniformly and absolutely convergent in
$L^\infty (D)$.
 We will freely use the fact that the series defining $f$ also converges absolutely and uniformly.


\noindent For $N>M>n_0(\epsilon)$ and $t\geq 0$ and $x\in D,$
\begin{eqnarray}
 |u_N(t,x)-u_M(t,x)| &=&
|\sum_{n=M}^{N}\phi_n(x)\bar{f}(n)h(t, \lambda_n)| \nonumber\\
&& \leq \sum_{n=M}^{N}|\phi_n(x)||\bar{f}(n)|< \epsilon,
\end{eqnarray}
since $h(t,\lambda)=\E(e^{-\lambda E_t})\leq 1$ for all $t\geq 0$
and $\lambda\geq 0$.

 This shows that the sequence $u_N(t,x)$ is a Cauchy sequence
in $L^\infty (D)$ and so has a limit in $L^\infty (D)$. Hence, the
series in (\ref{mittag-series}) is absolutely and uniformly
convergent. Also, it follows that $u(t,x)$ satisfies the boundary
conditions in  (\ref{frac-derivative-bounded-d}).


\vskip10pt \noindent {\bf Step 5.} Next we show  that  the
distributed-order fractional time derivative  and the Laplacian
$\Delta$ can be applied term by term in (\ref{mittag-series}). As
$h(t, \lambda)$ is bounded above by unity, we have
 \begin{eqnarray}
\left| \sum_{n=1}^\infty \bar f(n)h(t,\lambda_n)\Delta\phi_n(x)\right|&=& \left|\sum_{n=1}^\infty
\bar f(n)h(t,\lambda_n)
\lambda _n \phi_n(x)\right|\nonumber\\
 &\leq &\sum_{n=1}^{\infty}|\phi_n(x)|| \bar{f}(n)|\lambda_n<\infty,\label{upper-bound-derivatives}
 \end{eqnarray}
where the last inequality follows from the fact that the
eigenfunction expansion of $\Delta f$ converges absolutely and
uniformly. Then the series
$$
\sum_{n=1}^\infty \bar f(n)h(t,\lambda_n)\Delta\phi_n(x)
$$
is absolutely convergent in $L^\infty(D)$ uniformly in
$(0,\infty)$.  Since $h(t,\lambda)$ is an eigenfunction of the
distributed-order Caputo fractional derivative with
$\D^{(\nu)}h(t,\lambda)=-\lambda h(t,\lambda)$ (see
(\ref{dist-order-density-pde})), we have
 $$
\sum_{n=1}^\infty \bar f(n)\phi_n(x)\D^{(\nu)}h(t,\lambda_n)=
\sum_{n=1}^\infty \bar f(n)h(t, \lambda_n)\Delta\phi_n(x).
$$
As the two series are equal term-by-term and the series on the right converges absolutely and uniformly,
the series on the left
converges absolutely and uniformly too.

Now it is easy to check that
 the distributed-order fractional time derivative and Laplacian can be applied term by term in
(\ref{mittag-series}) to give
\begin{eqnarray*}
\D^{(\nu)}u(t,x) - \Delta
u(t,x)&&\\
=&&\sum_{n=1}^\infty \bar f(n)\left[\phi_n(x)\D^{(\nu)}
h(t,\lambda_n)-h(t, \lambda_n)\Delta\phi_n(x)\right]=0,
\end{eqnarray*}
so that the PDE in (\ref{frac-derivative-bounded-d}) is satisfied.
Thus, we conclude that $u$ defined by (\ref{mittag-series}) is a
classical (strong) solution to (\ref{frac-derivative-bounded-d}).



\noindent Further, using Lemma \ref{derivative-bound}, we get
\begin{eqnarray*}
\left|\frac{\partial u(t,x)}{\partial t}\right|&\leq &
\sum_{n=1}^\infty |\bar f(n)| \left|\frac{\partial h(t,
\lambda_n)}{\partial t}\right||\phi_n(x)| \\
 &\leq &  k(t)
\sum_{n=1}^\infty \lambda_n|\bar f(n)||\phi_n(x)|:=k(t)g(x).
\end{eqnarray*}
Since $\Delta f$ has absolutely and uniformly convergent series
expansion with respect to $\{\phi_n:\ n\in \RR{N}\},$ we have $g
\in L^\infty(D)$.




Thus, it follows, from the results obtained above,  that $u \in
\mathcal{H}_\Delta (D_\infty)\cap C_b(\bar D_\infty) $.

\vskip10pt \noindent{\bf  Step 6.}We next show that $u\in C^1(\bar D)$; this follows from the bounds
in \cite[Theorem 8.33]{gilbarg-trudinger}
 and the absolute and uniform convergence of the series defining
 $f$, namely,
 \begin{eqnarray}
 |\phi_n|_{1,\alpha; D}&\leq & C(1+\lambda_n)\sup_{D}|\phi_n(x)|,\label{uniform-eigenvalue-bounds}
 \end{eqnarray}
where $C=C(d, \lambda, \Lambda, \partial D)$ and $\lambda $ is the
constant in the definition of uniform ellipticity of $L$ in
\eqref{elliptic-bounds} and $\Lambda$ is the bound in
\eqref{L-uniform-bound}.  Here, $$|u|_{k,\alpha;
D}=\sup_{|\gamma|=k}[D^\gamma u]_{\alpha ,D}+ \sum_{j=0}^{k}
\sup_{|\gamma|=j}\sup_{D}|D^\gamma u|,\ \ k=0,1,2,\cdots ,$$ and
$$[D^\gamma u]_{\alpha ,D}=\sup_{x,y\in D, x\neq y}\frac{|D^\gamma u(x)- D^\gamma u(y)|}{|x-y|^\alpha}$$
are norms on $C^{k,\alpha}(\bar D)$. Hence,
\begin{eqnarray}
|u(t,.)|_{1,\alpha; D}&\leq & C \sum_{n=1}^\infty |\bar f(n)|h(t,\lambda_n)(1+\lambda_n)\sup_{D}|\phi_n(x)| \nonumber\\
&\leq &C \sum_{n=1}^\infty |\bar f(n)|(1+\lambda_n)\sup_{D}|\phi_n(x)| \nonumber\\
&\leq & C \sum_{n=1}^{\infty}\sup_{D}|\phi_n(x)|| \bar{f}(n)|+
C\sum_{n=1}^{\infty}\lambda_n \sup_{D}|\phi_n(x)|| \bar{f}(n)|
<\infty .\nonumber
 \end{eqnarray}



\vskip10pt \noindent{\bf  Step 7.}  We obtain here the stochastic
solution to (\ref{frac-derivative-bounded-d}), by inverting  the
$\phi_n$-Laplace transform. After showing the absolute and uniform
convergence of the series defining $u,$ we can use a
Fubini-Tonelli type argument to interchange order of summation and
integration in the following, together with \eqref{mittag-series}
and  \eqref{MMMc}, to get a stochastic representation of the
solution as
\begin{eqnarray}
u(t,x)&=&\sum_{n=1}^{\infty}\phi_n(x)\int_0^{\infty}\overline{[T_D(l)f]}(n)g(t,l)dl\nonumber\\
&=&\int_0^{\infty}\left[\sum_{n=1}^{\infty}\phi_n(x)\bar{f}(n)e^{-l\lambda_n}\right]g(t,l)dl\nonumber\\
&=&\int_{0}^{\infty}T_D(l)f(x)g(t,l)dl\nonumber\\
&=& E_x[f(X(E_t))I(\tau_D(X)>E_t)].
\end{eqnarray}
The last equality follows from a simple conditioning argument and
using (\ref{TDdef}).

\vskip10pt \noindent {\bf Step 8.} Finally, we prove the
uniqueness. Let $u_i, i=1, 2,$ be two solutions of
\eqref{frac-derivative-bounded-d} with initial data
$u_i(0,x)=f(x)$ and Dirichlet boundary condition $u_i(t,x)=0$ for
$x\in\partial D$. Then $U=u_1-u_2$ is also a solution of
\eqref{frac-derivative-bounded-d} with zero initial data and zero
boundary value.  Taking $\phi_n$-transform on both sides of
\eqref{frac-derivative-bounded-d} we get
$$
 \D^{(\nu)}{\bar{U}(t,n)} =
 -\lambda_n
\bar{U}(t,n), \ \ \bar{U}(0,n)=0,
$$
and then $\bar U(t,n)=0$ for all $t>0$ and all $n\geq 1$. This
implies that $U(t,x)=0$ in the sense of $L^2$ functions, since
$\{\phi_n: \ n\geq 1\}$ forms a complete orthonormal basis for
$L^2(D)$. Hence, $U(t,x)=0$ for all $t>0$ and almost all $x\in D$.
Since $U$ is a continuous function on $D$, we have $U(t,x)=0$ for
all $(t,x)\in [0,\infty)\times D,$ thereby proving the uniqueness.
\end{proof}

\begin{corollary}\label{subordination}
 The solution in Theorem (\ref{laplace-pde}) also has
the following representation:
\begin{eqnarray} u(t,x)&=&E_{x}[f(X(E_{t}))I(
\tau_D(X)> E_t)]=E_{x}[f(X(E_{t}))I( \tau_D(X(E))> t)].\nonumber
\end{eqnarray}
\end{corollary}

\begin{proof}  The argument is similar to \cite[Corollary 3.2]{m-n-v-aop} and so we only sketch the proof.
 Given a continuous stochastic process $X_t$ on $\rd$, and an interval $I\subset[0,\infty)$,
we denote $X(I)=\{X_u:u\in I\}$.  Since the domain $D$ is open and $X_u$ is continuous, it follows that
\[\{\tau_D(X)>t\}=\{X([0,t])\subset D\}.\]
Next note that, since $E_t$ is continuous and monotone
nondecreasing,
 $E([0,t])=[0,E_t]$.  Finally, we observe that
\begin{equation*}\begin{split}
\{\tau_D(X(E))>t\}&=\{X(E([0,t])\subset D\}\\
&=\{X([0,E_t])\subset D\}=\{\tau_D(X)>E_t\}
\end{split}\end{equation*}
which completes the proof.
\end{proof}

\begin{remark}
If we time-change Brownian motion $B(t)$ using a nondecreasing stable L\'evy process
$W_t$, then the conclusions of Corollary
\ref{subordination} do not hold, since the stable subordinator does
not have continuous sample paths; see, for example,  Song and
Vondra\u cek \cite{song-von}. In our case, killing the process $B(t)$ and
then applying the time change $E_t$ is the same as applying the time change and then killing, since the sample
paths of
$E_t$ are continuous.
\end{remark}

The next result establishes existence of strong solutions of
distributed-order fractional Cauchy problems \eqref{DOFCPdef} with
$L=\Delta$.

\begin{corollary}\label{strong-solution}
Let $f\in C^{2k}_c(D)$ be a $2k$-times continuously differentiable
function of compact support in D.  If $k>1+3d/4$, then
(\ref{frac-derivative-bounded-d}) has a classical (strong)
solution. In particular, if $f\in C^{\infty}_c(D)$, then the
solution of (\ref{frac-derivative-bounded-d}) is in $C^\infty(D).$
\end{corollary}

\begin{proof}
By Example 2.1.8 of \cite{davies},  $|\phi_n(x)|\leq
(\lambda_n)^{d/4}$. Also,  from Corollary 6.2.2 of
\cite{davies-d}, we have $\lambda_n\sim  n^{2/d}$.

\noindent Applying  Green's second identity k-times, we get
\begin{equation}\label{phi-trans-laplacian}
\overline{\Delta^k f}(n)=\int_D \Delta ^k f(x)\phi_n
(x)dx=(-\lambda_n)^k \bar{f}(n).
\end{equation}
Using  Cauchy-Schwartz inequality and the fact $f\in C^{2k}_c(D),$
we get
   $$\overline{\Delta^k f}(n)\leq \left[\int_D (\Delta^k f(x))^2dx\right]^{1/2}\left[\int_D (\phi_n(x))^2dx
   \right]^{1/2}=\left[\int_D (f^{2k}(x))^2dx\right]^{1/2}= c_k,$$
 where $c_k$ is  a constant independent of $n$.

\noindent This and  \eqref{phi-trans-laplacian} give
$|\bar{f}(n)|\leq c_k(\lambda_n)^{-k}$.

Since
$$
\Delta f(x)=\sum_{n=1}^\infty -\lambda_n\bar f (n) \phi_n(x),
$$
to get the absolute and uniform convergence of the series defining
$\Delta f$, we consider
\begin{equation*}\begin{split}
   \sum_{n=1}^{\infty}\lambda_n|\phi_n(x)|| \bar{f}(n)|&\leq \sum_{n=1}^{\infty} (\lambda_n)^{d/4+1}c_k(\lambda_n)^{-k}\\
   &\leq c_k\sum_{n=1}^{\infty} (n^{2/d})^{d/4+1-k}=c_k\sum_{n=1}^{\infty} n^{1/2+2/d-2k/d}
\end{split}\end{equation*}
 which is finite if $(\frac{2k}{d}-\frac{2}{d}-\frac{1}{2})>1$, i.e., if  $k>1+\frac{3}{4}d$.
\end{proof}

\begin{remark}
In an interval $(0,M)\subset \rr$, eigenfunctions and eigenvalues
are explicitly known.  Eigenvalues of the Laplacian on $(0,M)$ are
$(n\pi /M)^2$,
  and the corresponding eigenfunctions are $\sin(n\pi x/M),$ for $n=1,2,\cdots  .$ The form of the solution
  in  (\ref{mittag-series}) on a
bounded interval $(0,M)$ in $\rr$ was obtained by \cite{agrawal,
naber}. Agrawal \cite{agrawal} worked with single-order fractional
Cauchy problem with Dirichlet boundary conditions. Naber
\cite{naber} considered Dirichlet and Neumann boundary conditions
in one space dimension.
\end{remark}



\noindent Recall that  $D_\infty=(0,\infty )\times D$ and define
now
\begin{equation*}\begin{split}
\mathcal{H}_{L}(D_\infty)&= \{u:D_\infty\to \rr :\ \ L
u(t,x)\in C(D_\infty)\}; \\
\mathcal{H}_{L}^{b}(D_\infty)&= \mathcal{H}_{L}(D_\infty)
\cap \{u: |\partial_t u(t,x)|\leq k(t)g(x),\ \  g\in
L^\infty(D), \ t>0  \}
\end{split}\end{equation*}
where $k(t)$ is defined in \eqref{time-bound-derivative}. The
following result extends Theorem \ref{laplace-pde}
 to general uniformly elliptic second-order operators.



\begin{theorem}\label{FC-PDE-Lx}
   Let $D$ be a bounded domain with $\partial D \in
 C^{1,\alpha}$ for some $0<\alpha<1$, and suppose that $L$ is given by \eqref{unif-elliptic-op}
  with $a_{ij}\in C^{\alpha}(\bar D)$.  Let $\{X(t)\}$ be a continuous Markov process with generator $L$,
  and $T_D(t)$ the killed semigroup corresponding to the process $\{X(t)\}$ in $D$. Let $E_t$ be the inverse \eqref{Etdef}
of the subordinator $W_t$, independent of $\{X(t)\}$, with L\'evy
measure \eqref{psiWdef}. Suppose that
$\mu(d\beta)=p(\beta)d\beta$, as in Lemma \ref{derivative-bound},
and  $\D^{(\nu)}$ is the distributed-order fractional derivative
defined by \eqref{DOFDdef}.
Then, for any $f\in D(L_D)\cap C^1(\bar D)\cap C^2( D) $
  for which the eigenfunction expansion (of $Lf$) with respect to the complete orthonormal basis
  $\{\psi_n:\ n\geq 1\}$
  converges uniformly and absolutely,
the (classical) solution  of
\begin{eqnarray}
\D^{(\nu)}u(t,x)
&=&
L u(t,x),  \  \ x\in D, \ t\geq 0\label{frac-derivative-bounded-dd};\\
u(t,x)&=&0, \ x\in \partial D,\ t\geq 0; \nonumber\\
u(0,x)& =& f(x), \ x\in D,\nonumber
\end{eqnarray}
for  $u  \in
\mathcal{H}_{L}^{b}(D_\infty)\cap C_b(\bar D_\infty) \cap C^1(\bar D)$,
 is given by
\begin{eqnarray}
u(t,x)&=&E_{x}[f(X(E_{t}))I( \tau_D(X)> E_t)]=E_{x}[f(X(E_{t}))I( \tau_D(X(E))> t)]\nonumber\\
&=& \int_{0}^{\infty}T_D(l)f(x)g(t,l)dl= \sum_0^\infty
\bar{f}(n)\psi_n(x)h(t, \mu_n),\label{stoch-rep-L}
\end{eqnarray}\end{theorem}
where $h(t,\mu)=\E (e^{-\mu E_t})=\int_0^\infty e^{-x\mu}g(t,x)\,dx$
is the Laplace transform of $E_t$.

\begin{proof}
Let $u(t,x)=G(t)F(x)$ be a solution of
(\ref{frac-derivative-bounded-dd}).
Substituting in the PDE (\ref{frac-derivative-bounded-dd}) leads
to
$$
F(x)\D^{(\nu)} G(t)
= G(t)L F(x)
$$
and now dividing both sides by $G(t)F(x),$ we obtain
$$
\frac{\D^{(\nu)} G(t)}{G(t)} = \frac{L F(x)}{F(x)}=-\mu.
$$
That is,
\begin{equation}\label{time-pde}
\D^{(\nu)} G(t)=-\mu G(t), \ t>0;
\end{equation}
\begin{equation}\label{space-pde}
LF(x)=-\mu F(x), \ x\in D, \ F|_{\partial D}=0.
\end{equation}

\noindent Problem (\ref{space-pde}) is solved by an infinite
sequence of pairs $(\mu_n, \psi_n)$, $n \ge 1,$ where $0<
\mu_1<\mu_2\leq \mu_3 \leq \cdots$ is a sequence of numbers such
that $\mu_n\to\infty$, as $n\to\infty$, and $\psi_n$ is a sequence
of functions that form a complete orthonormal set in $L^2(D)$ (cf.
\eqref{eigen-eigen}).
 In particular, the initial function $f$ regarded as an element of $L^2(D)$ can be represented as
 \begin{equation}
 f(x)=\sum_{n=1}^\infty \bar f(n)\psi_n(x).
 \end{equation}
\noindent An application of the Parseval identity yields
\begin{equation}
 ||f||_{2,D}=\sum_{n=1}^\infty (\bar f(n))^2.
 \end{equation}

\noindent Using the $\mu_n$ determined by (\ref{space-pde}) and
recalling from Lemma \ref{eigenvalue-problem} that $h(t, \mu_n)$
solves  \eqref{time-pde} with $\mu=\mu_n$, we obtain
$$
G(t)=G_0(n)h(t, \mu_n),
$$
where $G_0(n)$ is selected to satisfy the initial condition $f$.
We will show that
\begin{equation}\label{formal-sol-L-1}
u(t,x)=\sum_{n=1}^{\infty}\bar{f}(n)h(t, \mu_n)\psi_n(x)
\end{equation}
solves the PDE (\ref{frac-derivative-bounded-dd}).

 Define approximate solutions of the form

\begin{equation}
 u_N(t,x)=\sum_{n=1}^{N}G_0(n)h(t, \mu_n)\psi_n(x),  \ G_0(n)=\bar{f}(n).
\end{equation}

\vskip10pt \noindent {\bf Step 1.} Following the proof of Theorem
\ref{laplace-pde}, it can be shown that the sequence
$\{u_N(t,\cdot )\}_{N\in \N}$ is a Cauchy sequence in $L^2(D)\cap
L^\infty(D)$, uniformly in $t\in [0,\infty)$.  Hence, the solution
to (\ref{frac-derivative-bounded-dd}) is given formally by
\begin{equation}\label{formal-sol-L}
u(t,x)=\sum_{n=1}^{\infty}\bar{f}(n)h(t, \mu_n)\psi_n(x).
\end{equation}

\vskip10pt \noindent {\bf Step 2.} It follows again, as in the
proof of Theorem {\ref{laplace-pde}},
 the series defining
$u$ and $ Lu$ converge absolutely and uniformly so that we can
apply the fractional
 time derivative and uniformly elliptic operator $L$ term by term to show that $u$ defined by
 (\ref{formal-sol-L}) is indeed a classical solution to \eqref{frac-derivative-bounded-dd}.

\vskip10pt \noindent {\bf  Step 3.} The stochastic representation
of the solution, as given in (\ref{stoch-rep-L}), also follows in
a similar manner and hence we omit the details.

\vskip10pt \noindent {\bf  Step 4.} We  have also a decay estimate
for $u$ as in Theorem \ref{laplace-pde}, namely,
$$
||u(\cdot, t)||_{2,D} \leq h(t,\mu_1)||f||_{2,D}.
$$

\noindent The uniqueness of the solution can be proved as
before.
\end{proof}

\section{Extensions and open questions}\label{sec5}

In this section, we derive some general conditions on the mixing
distribution $\mu(d\beta)$ in \eqref{DOFDdef} that are sufficient
to obtain classical solutions to the distributed-order fractional
Cauchy problem \eqref{DOFCPdef} on bounded domains, as in Theorems
\ref{laplace-pde} and \ref{FC-PDE-Lx}.  In particular, we remove
the assumption that $\mu(d\beta)=p(\beta)d\beta$ to allow atoms.
Then we discuss related literature, and some open questions.

\begin{lemma}
Suppose $\mu$ is a finite measure with  $supp (\mu)\subset(0,1)$
that satisfies \eqref{finite-mu-bound}. Assume also that
 $|\partial_th(t,\lambda)|\leq b(\lambda)k_e(t)$ for some functions $b$ and $ k_e$
satisfying the condition
\begin{equation}\label{e-convolution-k}
b(\lambda)\int_0^1\int_0^t\frac{k_e(s)ds}{(t-s)^\beta}d\mu(\beta)<\infty,
\end{equation}
for $t, \lambda >0$. Then $h(t,\lambda)=\E(e^{-\lambda E_t})$ is a classical solution of the eigenvalue problem

\begin{equation}\label{e-dist-order-density-pde}
\D^{(\nu)}h(t,\lambda)=-\lambda h(t, \lambda); \ \ h(0, \lambda)=1.
\end{equation}

\end{lemma}

\begin{proof}
The proof follows from Lemma \ref{eigenvalue-problem}, and the
fact that \eqref{e-convolution-k} is a sufficient condition for
$\D^{(\nu)} h(t, \lambda)$ to be defined as a classical function.
\end{proof}

Suppose that $|\partial_th(t,\lambda)|\leq b(\lambda)k_e(t)$, where
$b$ and $ k_e$  satisfy, in addition to  \eqref{e-convolution-k},
\begin{equation}\label{summability-k}
k_e(t)\sum_{n=1}^\infty b(\lambda_n)\bar{f}(n)|\phi_n(x)|<\infty.
\end{equation}
It is assumed here that the above series converges absolutely
and uniformly for $t>0$.

\noindent Let $\beta\in (0,1)$, $D_\infty=(0,\infty )\times D$
and define
\begin{equation*}\begin{split}
\mathcal{H}_{L}(D_\infty)&= \{u:D_\infty\to \rr :\ \ L
u(t,x)\in C(D_\infty)\};\\
\mathcal{H}_{L}^{b,e}(D_\infty)&= \mathcal{H}_{L}(D_\infty)
\cap \{u: |\partial_t u(t,x)|\leq k_e(t)g(x),\ \  g\in
L^\infty(D),  \ t>0  \}, \nonumber
\end{split}\end{equation*}
where  $k_e$ and $ b$ satisfy  \eqref{e-convolution-k} and
\eqref{summability-k}.
The following result extends Theorem \ref{FC-PDE-Lx}
 to allow atoms in the mixing measure $\mu(d\beta)$.


\begin{theorem}\label{e-FC-PDE-Lx}
   Let $D$ be a bounded domain with $\partial D \in
 C^{1,\alpha}$ for some $0<\alpha<1$, and suppose that $L$ is given by \eqref{unif-elliptic-op} with
 $a_{ij}\in C^{\alpha}(\bar D)$.  Let $\{X(t)\}$ be a continuous Markov process with generator $L$, and
 $T_D(t)$ be the killed semigroup corresponding to the process $\{X(t)\}$ in $D$. Let $E_t$ be the inverse
  \eqref{Etdef}
of the subordinator $W_t$, independent of $\{X(t)\}$, with L\'evy
measure \eqref{psiWdef}.
 Let $f\in D(L_D)\cap C^1(\bar D)\cap C^2( D) $
  for which the eigenfunction expansion (of $Lf$) with respect to the complete orthonormal basis
  $\{\psi_n:\ n\geq 1\}$
  converges uniformly and absolutely.
Then the (classical) solution  of
\begin{eqnarray}
\D^{(\nu)}u(t,x)
&=&
L u(t,x),  \  \ x\in D, \ t\geq 0\label{e-frac-derivative-bounded-dd};\\
u(t,x)&=&0, \ x\in \partial D,\ t\geq 0; \nonumber\\
u(0,x)& =& f(x), \ x\in D,\nonumber
\end{eqnarray}
for  $u  \in
\mathcal{H}_{L}^{b,e}(D_\infty)\cap C_b(\bar D_\infty) \cap C^1(\bar D)$, with the distributed order
fractional derivative $\D^{(\nu)}$ defined by \eqref{DOFDdef}, is given by
\begin{eqnarray}
u(t,x)&=&E_{x}[f(X(E_{t}))I( \tau_D(X)> E_t)]=E_{x}[f(X(E_{t}))I( \tau_D(X(E))> t)]\nonumber\\
&=& \int_{0}^{\infty}T_D(l)f(x)g(t,l)dl = \sum_0^\infty
\bar{f}(n)\psi_n(x)h(t, \mu_n),\label{e-stoch-rep-L}
\end{eqnarray}\end{theorem}
where $h(t,\mu)=\E (e^{-\mu E_t})=\int_0^\infty e^{-x\mu}g(t,x)\,dx$
is the Laplace transform of $E_t$.

\begin{proof}
The proof follows the same steps as in the proof of Theorems
\ref{laplace-pde} and \ref{FC-PDE-Lx} and  using the properties
\eqref{e-convolution-k} and  \eqref{summability-k} of the function
$k_e(t)$.
\end{proof}

\begin{remark}
Let $\mu(d\beta)=\sum_{j=1}^N
c_j^{\beta_j}(\Gamma(1-\beta_j))^{-1}\delta_{\beta_j}(\beta)d\beta$,
for  $0<\beta_1<\beta_2<\cdots <\beta_N<1$. In this case, the
subordinator is
 $W_t=\sum_{j=1}^N c_j W_t^{\beta_j}$ for independent stable subordinators $W_t^{\beta_j}$, for $j=1,\cdots, N$.
In this case, the functions  $k_e(t)$ and $b(\lambda)$ that
satisfy   \eqref{summability-k},  \eqref{e-convolution-k} and
$$
|\partial_th(t,\lambda)|\leq b(\lambda)k_e(t)
$$
are   $k_e(t)=(c_j^{\beta_j}\sin(\beta_j\pi))^{-1}(t^{\beta_j-1})$
for all $j=1,\cdots, N$ and $b(\lambda)=\lambda$, respectively.
The proof of this fact follows the same steps as in the proof of
equation (2.19) in \cite{koch3} using the properties of
$\mu(\beta)$. Hence, in this case Theorem \ref{e-FC-PDE-Lx}
applies and we have a classical solution of
\eqref{e-frac-derivative-bounded-dd} given by
\eqref{e-stoch-rep-L}.
\end{remark}

To conclude this paper, we discuss the related literature and some
open problems.  Given a uniformly bounded, strongly continuous
semigroup $T(t)$ with generator $L$ on some Banach space $H$, the
Cauchy problem
 $\partial u(t,x)/\partial t=L u(t,x)$ with $u(0,x)=f(x)$ has solution $u(t,x)=T(t)f(x)$ for any initial
 condition $f\in H$ \cite{ABHN}.  One important special case is the pseudo-differential operator
\begin{equation}\begin{split}\label{jumpgen}
Lu(t,x)
&=\sum_{i,j=1}^{d}a_{ij}(x)\frac{ \partial^2 u(t,x)}{\partial
x_i\partial  x_j}+\sum_{i=1}^d b_i(x)\frac{\partial u(t,x)}{\partial
x_i}\\
&+\int_{y\neq 0}\left(u(t,x-y)-u(t,x)+\frac{\sum_{i=1}^d \frac{\partial u(t,x)}{\partial
x_i}y_i}{1+\sum_{i=1}^d y_i^2}\right)\,\phi(x,dy)
\end{split}\end{equation}
that appears in the backward equation of a Markov process $X(t)$ \cite{Jacob,schilling}.  The probability
distribution of the Markov process $X(t)$ solves the forward equation, which is the Cauchy problem with the
adjoint of the generator $L$.  The integral term in \eqref{jumpgen} represents a jump diffusion
(e.g., a stable process).  For stable generators, the explicit connection with stochastic differential
equations driven by a stable L\'evy process was established by Zhang et al. \cite{ParticleTracking}
and Chakraborty \cite{Chakraborty}.  In that case, the integral term in \eqref{jumpgen} can be written in terms
of fractional derivatives in the space variable.

The solution \eqref{solution-laplacian-frac-cauchy} to the fractional Cauchy problem
 $\partial^\beta u(t,x)/\partial t^\beta=L u(t,x)$ for $0<\beta<1$ was established by Baeumer and Meerschaert
  \cite{fracCauchy} in the general Banach space setting.  Baeumer et al.\ \cite{bmn-07} and Nane \cite{nane}
  specialized to Markov processes with generator \eqref{jumpgen}, and established a connection to iterated
  Brownian motion \cite{allouba1,burdzy1,burdzy2, deblassie}.   Hahn et al.\ \cite{h-k-umarov} developed the
  connection with stochastic differential equations driven by a time-changed L\'evy process $X(E_t)$ for
  generators \eqref{jumpgen}, so that their result includes jump diffusions on $\rd$.  Their results extend
  those of \cite{Chakraborty, ParticleTracking} to distributed-order fractional Cauchy problems on $\rd$.
  Hahn et al.\ \cite{h-k-umarov} also give the integral solution \eqref{solution-laplacian-frac-cauchy}
  as in \cite{fracCauchy,bmn-07}.  Kochubei \cite{koch3} provides strong solutions of distributed order
  fractional Cauchy problems on $\rd$ in the case $L=\Delta$.  Meerschaert and Scheffler \cite{M-S-triangular}
  discuss generalized Cauchy problems of the form
  $\psi_W(\partial/\partial t)u(t,x)=L u(t,x)+\delta(x)\psi_W(t,\infty)$, where $\psi_W(s)$ is the
  Laplace exponent of a nondecreasing L\'evy process (subordinator) whose L\'evy measure $\phi_W$ has
  infinite total mass, and $L$ is the generator of another L\'evy process.  This reduces to the distributed
  order fractional Cauchy problem \eqref{DOFCPdef} in the special case when \eqref{phiWdef} holds.
  As in Section \ref{sec3}, the paper \cite{M-S-triangular} shows that the density $u(t,x)$ of the CTRW scaling
  limit $A(E_t)$ solves the generalized Cauchy problem, when $E_t$ is the inverse of the subordinator $W_t$
  with $\E[e^{-sW_t}]=e^{-t\psi_W(s)}$.  Strong solutions of generalized Cauchy problems on $\rd$, including
  fractional or distributed-order Cauchy problems, seem to be an open problem.

For bounded domains, the general results of \cite{fracCauchy}
remain valid, so that the solution formula
\eqref{solution-laplacian-frac-cauchy} still holds in the
appropriate Banach space.  The results of this paper provide
strong solutions in that case, so long as $L$ generates a
diffusion without jumps.  To the best of our knowledge, the
construction of strong solutions for jump diffusions remains a
challenging open problem. Eigenvalue expansions
 can be found explicitly in some special cases.  The main technical difficulty is to obtain regularity of the
 eigenfunctions, or at least sharp bounds, for the generator \eqref{jumpgen} in the case of jump diffusions on
 bounded domains.  See Chen et al. \cite{chen-kim-song} for a recent study on this problem.  One explicit example
 is to take $L=-(-\Delta)^{\alpha/2}$ for $0<\alpha<2$, the classical fractional power of the Laplacian
 \cite{Hi-Ph}, which generates a spherically symmetric stable L\'evy process.  This results from \eqref{jumpgen}
 with $a=b=0$ and $\phi(x,dy)=C_{d,\alpha} \|y\|^{-\alpha-1}dy$, where $C_{d,\alpha}$ is a constant that depends
 on the stable index $\alpha$ and the dimension $d$ of the space, see for example \cite{multiFADE}.  This is a
 type of fractional derivative in space, called the Riesz fractional derivative of order $\alpha$.  Some results
 for this case are available in Chen and Song \cite{chenSong}, Chen et al. \cite{chen-kim-song2} and  Song and
 Vondra\u cek \cite{song-von}.

\end{document}